\documentclass[oneside]{amsart}
\usepackage[utf8]{inputenc}
\usepackage{amsmath,amstext,amsopn,amssymb, amsfonts, amsthm, mathtools}
\usepackage[dvipsnames]{xcolor}
\usepackage[marginparwidth=2.5cm]{geometry}
\usepackage{float}
\usepackage{bbm}
\usepackage{pdfsync}
\usepackage{tikz}
\usepackage{marginnote}
\usepackage{enumitem}
\usepackage{subfig}
\usepackage{caption}
\usepackage{graphicx}
\usepackage{mathrsfs}
\usepackage[ruled,vlined,algosection]{algorithm2e}
\DeclareMathAlphabet{\mathpzc}{OT1}{pzc}{m}{it}

\newcommand{\ultriangle}{%
  \mathord{\tikz[baseline=-0.15ex,scale=0.13]
    \draw[line width=0.5pt] (0,0) -- (0,1) -- (1,1) -- cycle;}%
}

\usepackage[pdftex,bookmarks,colorlinks,breaklinks,pagebackref]{hyperref}
\usepackage[abbrev]{amsrefs}

\definecolor{dullmagenta}{rgb}{0.4,0,0.4}   
\definecolor{darkblue}{rgb}{0,0,0.4}
\definecolor{darkgreen}{rgb}{0,0.4,0}

\definecolor{figplum}{HTML}{8B5A9F}
\definecolor{figteal}{HTML}{008C95}
\definecolor{figslate}{HTML}{475569}
\definecolor{cbvermillion}{RGB}{213,94,0}
\definecolor{cbsky}{RGB}{86,180,233}

\hypersetup{
  linkcolor=darkblue,
  citecolor=blue,
  filecolor=dullmagenta,
  urlcolor=darkblue
}

{}

\newtheorem{definition}{Definition}
\newtheorem*{definition*}{Definition}

\newtheorem{theorem}{Theorem}
\newtheorem*{theorem*}{Theorem}

\newtheorem*{conjecture*}{Conjecture}

\newtheorem*{question*}{Question}

\newtheorem{lemma}[theorem]{Lemma}
\newtheorem*{lemma*}{Lemma}

\newtheorem{corollary}[theorem]{Corollary}
\newtheorem*{corollary*}{Corollary}

\theoremstyle{definition}

\newtheorem*{remark*}{Remark}

\theoremstyle{plain}

\newtheorem*{example*}{Example}

\numberwithin{equation}{section}
\numberwithin{theorem}{section}

\makeatletter
\newcommand{\customlabel}[2]{
   \protected@write \@auxout {}{
     \string \newlabel {#1}{{#2}{\thepage}{#2}{#1}{}} }
   \hypertarget{#1}{#2}
}
\makeatother

\def\XXint#1#2#3{{\setbox0=\hbox{$#1{#2#3}{\int}$}
     \vcenter{\hbox{$#2#3$}}\kern-.5\wd0}}

\newcommand{\sh}{\operatorname{sh}}

\newcommand{\ind}{\mathbbm{1}}
\newcommand{\mass}{\operatorname{Mass}}

\newcounter{cmt}
\newcommand{\comment}[2]{%
  \stepcounter{cmt}%
  \ifodd\value{cmt}\normalmarginpar\else\reversemarginpar\fi
  \marginnote{%
    \tikz\node[draw=blue, rounded corners=2pt, inner sep=4pt,
               text width=\dimexpr\marginparwidth-12pt\relax,
               align=left, font=\footnotesize, text=blue]{#2};%
  }[\baselineskip]%
  {\color{red}#1}%
}
\newcommand{\commentnote}[1]{%
  \stepcounter{cmt}%
  \ifodd\value{cmt}\normalmarginpar\else\reversemarginpar\fi
  \marginnote{%
    \tikz\node[draw=blue, rounded corners=2pt, inner sep=4pt,
               text width=\dimexpr\marginparwidth-12pt\relax,
               align=left, font=\footnotesize, text=red]{#1};%
  }%
}

\makeindex
\begin{document}

\title{Improved weighted bounds for the strong maximal function}
\author{Sheldy Ombrosi}

\address{Departamento de Matem\'atica e Instituto de Matemática,
Universidad Nacional del Sur (UNS)--CONICET,
Bah\'ia Blanca, Argentina}

\email{sombrosi@uns.edu.ar}
\author{Guillermo Rey}
\address{Universidad Autónoma de Madrid}
\email{guillermo.rey@uam.es}

\thanks{Research supported in part by grant RYC2024-051323-I, funded by MICIU/AEI/10.13039/501100011033.}
\begin{abstract}
We improve the exponents of the $A_p$ constant in the
strong and weak-type weighted bounds for the strong maximal function,
for every $1<p<\infty$ and every dimension $d\geq2$.
In dimension two, the strong and weak $L^2(w)$ exponents improve
from $2$ to $7/4$ and from $3/2$ to $5/4$, respectively.
The proof exploits the geometry of pairwise intersections of
antichains of dyadic rectangles.
\end{abstract}

\maketitle


\section{Introduction}

Let $\mathcal{M}_{\mathcal{R}}$ be the strong maximal function in
$\mathbb{R}^d$, that is
\begin{align*}
  \mathcal{M}_{\mathcal{R}}f = \sup_{R \in \mathcal{R}} \frac{\ind_R}{|R|}
    \int_R |f(x)| \, dx,
\end{align*}
where $\mathcal{R}$ consists of all
$d$-dimensional axis-parallel rectangles,
i.e.: cartesian products of $d$ one-dimensional
intervals.

It is well-known that $\mathcal{M}_{\mathcal{R}}$ is bounded
on $L^p(\mathbb{R}^d)$ for $1 < p \leq \infty$.
This follows by a standard iteration argument
using the observation that
\begin{align*}
  \mathcal{M}_{\mathcal{R}}f \leq (M_1 \circ \dots \circ M_d)f,
\end{align*}
where $M_i$ are the one-dimensional maximal
functions along the $i$-th direction.

When considering weighted estimates, the same
argument yields, for $1<p<\infty$ and $w\in A_p(\mathcal{R})$,
\begin{align} \label{intro:iteration}
  \|\mathcal{M}_{\mathcal{R}}f\|_{L^p(w)} \lesssim_{p,d} [w]_{A_p(\mathcal{R})}^{\frac{d}{p-1}} \|f\|_{L^p(w)},
\end{align}
where $[w]_{A_p(\mathcal{R})}$ is defined by
\begin{align*}
  [w]_{A_p(\mathcal{R})} = \sup_{R \in \mathcal{R}} \langle w \rangle_R
  \langle \sigma \rangle_R^{p-1} \qquad 1 < p < \infty.
\end{align*}
Here $\sigma$ denotes the dual weight $\sigma = w^{1-p'}$ and
$\langle f \rangle_R$ denotes the average of $f$ over $R$:
$\langle f \rangle_R = \frac{1}{|R|}\int_R f$.
Standard arguments show that $w \in A_p(\mathcal{R})$ is a necessary
condition for boundedness on $L^p(w)$, but the sharp exponent is not known.
See \cite{Luque2014} for an enlightening discussion on this topic.

In a recent preprint \cite{Lerner2026}, A. Lerner has shown that
a linear $A_2(\mathcal{R})$ bound fails. More precisely, in
$\mathbb{R}^2$ he obtains the lower bound
\begin{align*}
  \|\mathcal{M}_{\mathcal{R}}\|_{L^2(w) \to L^2(w)} \gtrsim
  [w]_{A_2(\mathcal{R})} \sqrt{\log [w]_{A_2(\mathcal{R})}}
\end{align*}
for a sequence of weights $w$ whose $A_2(\mathcal{R})$
characteristic tends to $\infty$. 

Up to now and to our knowledge, no upper bound with an exponent smaller than $\frac{d}{p-1}$ is known.
The main purpose of this paper is to show that \eqref{intro:iteration} can be improved. More specifically, we obtain the following result:

\begin{theorem} \label{MainTheorem}
  Let $1<p<\infty$, $d\geq 2$, and $w \in A_p(\mathcal{R})$. Then
  \begin{align*}
    \|\mathcal{M}_{\mathcal{R}}f\|_{L^p(w)}
    \lesssim_{p,d} [w]_{A_p(\mathcal{R})}^{\alpha}
    \|f\|_{L^p(w)},
  \end{align*}
  where
  \begin{align*}
    \alpha = \frac{d}{p-1} - \frac{\Big\lfloor \dfrac{d}{2}\Big\rfloor}{2p(p-1)}.
  \end{align*}
\end{theorem}
In particular, when $d=2$ and $p=2$, the strong-type exponent improves
from $2$ to $7/4$.

The main idea in the proof of this theorem, instead of iterating the
one-dimensional maximal function, is to exploit the geometry of pairwise
intersections of antichains of two-dimensional dyadic rectangles.
We first use the $\frac13$-trick to reduce to the dyadic setting,
in particular one can dominate the strong maximal function by a sum
of a finite number of maximal functions associated to rectangles belonging
to different dyadic grids.
See, for example, \cite{LiPipherWard2015}*{Theorem 6.1(ii)}.
In Section~2, the estimates are formulated for fixed product dyadic grids
and their associated dyadic weight characteristics. They are uniform
in the choice of grid. Since each dyadic characteristic is bounded by
$[w]_{A_p(\mathcal{R})}$, the dyadic reduction changes only the
dimensional constants.

Instead of directly attacking the strong $L^p$ bound,
we follow Buckley's proof of the sharp weighted bounds for the (one-parameter) maximal function \cite{Buckley1993}: we first prove
a weighted weak-type inequality and then,
using the reverse H\"older property for strong $A_p(\mathcal{R})$ weights found in \cite{LPR2017}, interpolate this to obtain Theorem \ref{MainTheorem}.

The benefit of working in the weak-type setting is that we can assume that the covering
of the level set of the strong maximal function is an \emph{antichain},
that is: no rectangle is contained in another.
This, together with a selection theorem inspired by C\'ordoba
and Fefferman's proof of the strong maximal theorem in \cite{CF1975},
allows us to work with a \eqref{P2} family of dyadic rectangles while only paying by a multiplicative factor of $[w]_{A_p(\mathcal{R})}$ in the weighted
measure of the level set.

We then show that the collection of all pairwise intersections of a \eqref{P2} family of dyadic two-dimensional rectangles is sparse.
These pairwise intersections have disjoint parts which carry more
geometric information than what one can usually assume from a generic
sparse collection.

Finally, grouping the coordinates in pairs gives the improvement in every dimension.

The weak-type result mentioned above is the following.
\begin{theorem}\label{MainWeakTheorem}
  Let $1<p<\infty$, $d\geq 2$, and $w\in A_p(\mathcal{R})$. Then
  \begin{align*}
    \|\mathcal{M}_{\mathcal{R}}f\|_{L^{p,\infty}(w)}
    \lesssim_{p,d} [w]_{A_p(\mathcal{R})}^{\beta}
    \|f\|_{L^p(w)},
  \end{align*}
  where
  \begin{align*}
    \beta = \frac1p + \frac{d-1}{p-1}
      - \frac{\Big\lfloor \dfrac{d}{2} \Big\rfloor}{2p(p-1)}.
  \end{align*}
\end{theorem}
For comparison, Luque, P\'erez, and Rela
\cite{LPR2017}*{equation (4.10)} prove the upper bound
\begin{align*}
  \|\mathcal{M}_{\mathcal{R}}f\|_{L^{p,\infty}(w)}
  \lesssim_{p,d}
  [w]_{A_p(\mathcal{R})}^{\frac1p+\frac{d-1}{p-1}}
  \|f\|_{L^p(w)}.
\end{align*}
Thus Theorem~\ref{MainWeakTheorem} improves the exponent by
$\lfloor d/2\rfloor/(2p(p-1))$ in every dimension $d\geq2$.
In particular, when $d=2$ and $p=2$, the weak-type exponent improves
from $3/2$ to $5/4$.

These theorems are proved at the end of the next section.

\section{The weak-type estimate and its applications}
In this section we work with a fixed but arbitrary family
of one-dimensional dyadic grids $\{\mathcal{D}_i\}_{i=1}^d$.
The collection $\mathcal{R}_d$ of dyadic rectangles with respect to this
family is the one formed by the cartesian products of dyadic intervals in each grid
\begin{align*}
  \mathcal{R}_d = \Bigg\{ \prod_{i=1}^d I_i:\, I_i \in \mathcal{D}_i \Bigg\}.
\end{align*}

One can define the strong $A_p$ constant associated to such a family
as follows
\begin{align*}
  [w]_{A_p(\mathcal{R}_d)} = \sup_{R \in \mathcal{R}_d}
    \langle w \rangle_R \langle \sigma \rangle_R^{p-1},
\end{align*}
where, as before, $\sigma = w^{1-p'}$.

For $0<\eta\leq1$, a family $\mathcal{F}$ is called
\emph{$\eta$-sparse} if there are pairwise disjoint measurable sets
$E_R\subseteq R$, indexed by $R\in\mathcal{F}$, such that
$|E_R|\geq\eta|R|$ for every $R\in\mathcal{F}$.
A stronger notion was introduced by C\'ordoba and Fefferman in
\cite{CF1975}. To define it, let us introduce first the \emph{shadow}
of a collection.

Given a collection $\mathcal{E}$ of dyadic rectangles, define its
shadow by
\begin{align*}
  \sh(\mathcal{E}) := \bigcup_{R \in \mathcal{E}} R.
\end{align*}

\begin{definition}
  We say that a family $\mathcal{F}$ is \eqref{P2} if, for every $R\in\mathcal{F}$,
  \begin{align} \tag{$P_2$} \label{P2}
    |R \cap \sh(\mathcal{F} \setminus \{R\})| \leq \frac{1}{2}|R|.
  \end{align}
\end{definition}

C\'ordoba and R. Fefferman showed in \cite{CF1975} that every family $\mathcal{E}$ of rectangles contains
a \eqref{P2} subcollection $\mathcal{F}$ such that
\begin{align*}
|\sh(\mathcal{E})| \lesssim |\sh(\mathcal{F})|.
\end{align*}
We shall need a weighted variant of this selection,
one in which we obtain
\begin{align*}
w(\sh(\mathcal{E})) \leq C_w w(\sh(\mathcal{F})).
\end{align*}
The original proof in \cite{CF1975} proceeded in two passes.
In each pass one picks up a multiplicative constant comparable
to the weak-type norm of the strong maximal function.
In the weighted setting this produces a suboptimal power of the $A_p(\mathcal{R}_d)$ constant.
In \cite{Rey2026} a \emph{one-pass} method was introduced which bypasses
one of the multiplicative losses described above.
For the sake of completeness, we reproduce it here.
We note in passing that the following result does not need the
rectangles to be dyadic.
\begin{theorem} \label{p2_variational_selection}
  Let $\mathcal{E}$ be a finite collection of rectangles.
  There exists a \eqref{P2} subcollection $\mathcal{G} \subseteq \mathcal{E}$ such
  that for every $R \in \mathcal{E} \setminus \mathcal{G}$ we have
  \begin{align*}
    |R \cap \sh(\mathcal{G})| \geq \frac{1}{2}|R|.
  \end{align*}
\end{theorem}
\begin{proof}
  For every $\mathcal{F}$,
  let $\mass(\mathcal{F}) = \sum_{R \in \mathcal{F}} |R|$,
  and define the functional
  \begin{align*}
    \Phi(\mathcal{F}) = \frac{1}{2} \mass(\mathcal{F}) - |\sh(\mathcal{F})|.
  \end{align*}
  
  Among the \emph{finite} collection of all subcollections of $\mathcal{E}$,
  choose one that minimizes $\Phi$, call it $\mathcal{G}$.

  Let $R \in \mathcal{G}$ and set $\mathcal{G}' = \mathcal{G} \setminus \{R\}$.
  Observe that we have
  \begin{align*}
    \mass(\mathcal{G}) = \mass(\mathcal{G}') + |R|
    \quad\text{and}\quad
    |\sh(\mathcal{G})| = |\sh(\mathcal{G}')| + |R \setminus \sh(\mathcal{G}')|.
  \end{align*}
  As a consequence, we can write $\Phi(\mathcal{G}) = \Phi(\mathcal{G}') + \frac{1}{2}|R| - |R \setminus \sh(\mathcal{G}')|$.
  Using the minimality of $\mathcal{G}$:
  \begin{align*}
    \Phi(\mathcal{G}) = \Phi(\mathcal{G}') + \frac{1}{2}|R| - |R \setminus \sh(\mathcal{G}')| \leq \Phi(\mathcal{G}').
  \end{align*}
  In other words
  \begin{align*}
    |R \cap \sh(\mathcal{G}')| \leq \frac{1}{2}|R|,
  \end{align*}
  which is the \eqref{P2} condition.

  Now let $R \in \mathcal{E} \setminus \mathcal{G}$.
  Comparing $\mathcal{G}$ with $\mathcal{G} \cup \{R\}$, minimality gives
  \begin{align*}
    0 &\leq \Phi(\mathcal{G} \cup \{R\}) - \Phi(\mathcal{G}) \\
      &= \frac{1}{2}|R| - |R \setminus \sh(\mathcal{G})| \\
      &= |R \cap \sh(\mathcal{G})| - \frac{1}{2}|R|.
  \end{align*}
  Thus $|R \cap \sh(\mathcal{G})| \geq \frac{1}{2}|R|$, as required.
\end{proof}

\begin{corollary}\label{p2_weighted_selection}
  Let $\mathcal{E}$ be a finite collection of dyadic rectangles in dimension two, and suppose $w \in A_p(\mathcal{R}_2)$.
  There exists a subcollection $\mathcal{G} \subseteq \mathcal{E}$
  satisfying the \eqref{P2} condition and
  \begin{align*}
    w(\sh(\mathcal{E})) \leq 2^{1+2p}[w]_{A_p(\mathcal{R}_2)} w(\sh(\mathcal{G})).
  \end{align*}
\end{corollary}
\begin{proof}
  Let $\mathcal{F}$ consist of the elements in $\mathcal{E}$ which are maximal by inclusion, thus $\mathcal{F}$ is an antichain and
  \begin{align*}
    \sh(\mathcal{E}) = \sh(\mathcal{F}).
  \end{align*}

  Now let $\mathcal{G}$ be the \eqref{P2} subcollection of $\mathcal{F}$
  provided by Theorem \ref{p2_variational_selection},
  and let $\mathcal{B} = \mathcal{F} \setminus \mathcal{G}$.
  Fix $R \in \mathcal{B}$
  and define
  \begin{align*}
    \mathcal{G}_i = \{S \in \mathcal{G}: S \cap R \neq \emptyset \text{ and }  \pi_i(S) \subseteq \pi_i(R)\}.
  \end{align*}
  Since $\mathcal{F}$ is an antichain of two-dimensional rectangles,
  every rectangle $S \in \mathcal{G}_1$ also satisfies
  \begin{align*}
    \pi_2(S) \supseteq \pi_2(R).
  \end{align*}
  Since all the rectangles involved are dyadic, we have
  \begin{align*}
    \frac{1}{2}|R| \leq |R \cap \sh(\mathcal{G})| &\leq |R \cap \sh(\mathcal{G}_1)|
    + |R \cap \sh(\mathcal{G}_2)|.
  \end{align*}
  Therefore, there exists an $i \in \{1,2\}$ such that
  \begin{align*}
    \frac{1}{4} |R| \leq |R \cap \sh(\mathcal{G}_i)|.
  \end{align*}
  Let $M_i$ be the dyadic maximal function along the $i$-th direction:
  \begin{align*}
    M_1 f(x,y) &:= \sup_{I \in \mathcal{D}_1} \ind_I(x) \frac{1}{|I|} \int_I |f(s,y)| \, ds, \\
    M_2 f(x,y) &:= \sup_{J \in \mathcal{D}_2} \ind_J(y) \frac{1}{|J|} \int_J |f(x,t)| \, dt.
  \end{align*}
  If $\frac{1}{4} |R| \leq |R \cap \sh(\mathcal{G}_i)|$, then
  \begin{align*}
    R \subseteq \{ M_i(\ind_{\sh(\mathcal{G}_i)}) \geq \frac{1}{4} \};
  \end{align*}
  therefore
  \begin{align*}
    \sh(\mathcal{B}) \subseteq \{ M_1(\ind_{\sh(\mathcal{G})}) \geq \frac{1}{4} \} \cup
      \{ M_2(\ind_{\sh(\mathcal{G})}) \geq \frac{1}{4} \}.
  \end{align*}
  Since $M_i(\ind_{\sh(\mathcal{G})})=1$ on $\sh(\mathcal{G})$
  for $i\in\{1,2\}$, the same union also contains $\sh(\mathcal{F})$.
  The dyadic $A_p$ characteristics of almost every coordinate slice of $w$
  are bounded by $[w]_{A_p(\mathcal{R}_2)}$. Applying the one-dimensional
  dyadic weak-type bounds on these slices and integrating, we obtain
  \begin{align*}
    w(\sh(\mathcal{E})) = w(\sh(\mathcal{F}))
    &\leq \sum_{i=1}^2 w\big(\{M_i(\ind_{\sh(\mathcal{G})})\geq 1/4\}\big) \\
    &\leq 2^{1+2p}[w]_{A_p(\mathcal{R}_2)} w(\sh(\mathcal{G})),
  \end{align*}
  and the claim follows. The directional strip structure is illustrated
  in Figure~\ref{fig:directional_strips}.
\end{proof}

\begin{figure}[htpb]
  \centering
  \begin{tikzpicture}[x=0.5cm,y=0.5cm,
    lbl/.style={font=\small,inner sep=1.25pt},
    ptr/.style={->,thick,shorten >=2pt}]
    \begin{scope}
      \node[lbl] at (2.5,9) {(a) Vertical strips};
      \foreach \x/\h in {0/6,2/8}{
        \fill[figteal,fill opacity=0.3] (\x,0) rectangle (\x+1,\h);
        \draw[figteal,very thick] (\x,0) rectangle (\x+1,\h);
      }
      \draw[gray!80,very thick] (0,0) rectangle (4,4);
      \node[gray!80,lbl] at (4.4,4.3) {$R$};
      \node[figteal,lbl] (G1) at (4.7,6.5) {$\mathcal{G}_1$};
      \draw[ptr,figteal!65] (G1) to[bend left=15] (2.55,6.1);
      \draw[gray!65,thin,<->] (-0.8625,1.4) -- (4.8625,1.4);
      \draw[figteal,very thick] (0,1.4) -- (1,1.4);
      \draw[figteal,very thick] (2,1.4) -- (3,1.4);
      \fill (3.55,1.4) circle[radius=1.3pt];
      \node[lbl,anchor=west] (XY1) at (5.1,2.7) {$(x,y)$};
      \draw[ptr,gray!65] (XY1) to[bend right=15] (3.55,1.4);
      \node[lbl,gray!75] at (2.4,-0.75) {horizontal slice};
      \node[lbl] at (2.4,-1.65)
        {$M_1(\ind_{\sh(\mathcal{G}_1)})(x,y)\geq\frac12$};
    \end{scope}
    \begin{scope}[shift={(9,0)}]
      \node[lbl] at (4,9) {(b) Horizontal strips};
      \foreach \y/\w in {0/6,2/8}{
        \fill[figplum,fill opacity=0.3] (0,\y) rectangle (\w,\y+1);
        \draw[figplum,very thick] (0,\y) rectangle (\w,\y+1);
      }
      \draw[gray!80,very thick] (0,0) rectangle (4,4);
      \node[gray!80,lbl] at (4.45,4.3) {$R$};
      \node[figplum,lbl] (G2) at (6.75,4.6) {$\mathcal{G}_2$};
      \draw[ptr,figplum!65] (G2) to[bend left=15] (6.4,2.55);
      \draw[gray!65,thin,<->] (1.4,-0.8625) -- (1.4,4.8625);
      \draw[figplum,very thick] (1.4,0) -- (1.4,1);
      \draw[figplum,very thick] (1.4,2) -- (1.4,3);
      \fill (1.4,3.55) circle[radius=1.3pt];
      \node[lbl] (XY2) at (2.75,5.35) {$(x,y)$};
      \draw[ptr,gray!65] (XY2) to[bend left=15] (1.4,3.55);
      \node[lbl,gray!75] at (4,-0.75) {vertical slice};
      \node[lbl] at (4,-1.65)
        {$M_2(\ind_{\sh(\mathcal{G}_2)})(x,y)\geq\frac12$};
    \end{scope}
  \end{tikzpicture}
  \caption{The directional strip structure in the weighted selection argument.
    Members of $\mathcal{G}_1$ span the height of $R$, while members of
    $\mathcal{G}_2$ span its width. In this example, each family occupies
    half of $R$. Every parallel slice has the same covered proportion,
    giving the displayed lower bounds throughout $R$.
    The unequal extensions beyond $R$ are drawn schematically.}
  \label{fig:directional_strips}
\end{figure}
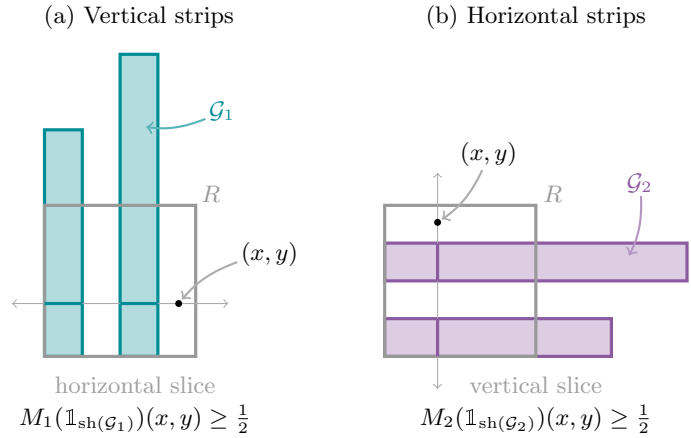

For the next theorem, and further arguments later, we will use the following
partial order on dyadic rectangles in $\mathcal{R}_2$:
\begin{align*}
  R \preceq S \iff \pi_1(R) \subseteq \pi_1(S)
  \quad\text{and}\quad \pi_2(R) \supseteq \pi_2(S).
\end{align*}
We write $R \prec S$ if $R \preceq S$ and $R \neq S$.
If $R$ and $S$ intersect and belong to an antichain with respect to inclusion,
then dyadicity implies that $R \preceq S$ or $S \preceq R$.
Indeed, their projections are nested in each coordinate, and the antichain
condition forces the two containments to point in opposite directions.
For distinct such rectangles, both containments are strict.

For a collection $\mathcal{F}$ of dyadic rectangles, $i \in \{1,2\}$,
and $R \in \mathcal{F}$, define
\begin{align*}
  \xi_i(R;\, \mathcal{F}) = \{S \in \mathcal{F} \setminus \{R\}:\,
    \pi_i(S) \supseteq \pi_i(R)\}.
\end{align*}
With this, we can define the \emph{directionally good} parts of $R$:
\begin{align*}
  E_i(R;\, \mathcal{F}) = R \setminus \sh(\xi_i(R;\, \mathcal{F})).
\end{align*}

\begin{theorem}\label{pwi:sparse_intersections}
  Let $\mathcal{G}$ be a finite \eqref{P2} collection of dyadic rectangles in
  $\mathcal{R}_2$, and define the set of ordered pairs
  \begin{align*}
    \mathcal{G}^{\ultriangle} := \{(R,S) \in \mathcal{G}^2:\, R \preceq S\}.
  \end{align*}
  Define also the collection of nonempty pairwise intersections
  \begin{align*}
    \mathcal{G}^{\land} = \{R \cap S:\, (R,S) \in \mathcal{G}^{\ultriangle}\}.
  \end{align*}
  Then $\mathcal{G}^{\land}$ is $1/4$-sparse. More precisely, for
  $(R,S) \in \mathcal{G}^{\ultriangle}$, define
  \begin{align*}
    E(R,S) := E_2(R;\, \mathcal{G}) \cap E_1(S;\, \mathcal{G}).
  \end{align*}
  The sets $\{E(R,S):\, (R,S) \in \mathcal{G}^{\ultriangle}\}$
  are pairwise disjoint, cover $\sh(\mathcal{G})$, and satisfy
  \begin{align} \label{pwi:large_hiding_spot}
    |E(R,S)| \geq \frac{1}{4}|R \cap S|.
  \end{align}
  Furthermore, for $(R,S) \in \mathcal{G}^{\ultriangle}$,
  \begin{align} \label{pwi:cover}
    R \cap S = \bigsqcup_{\substack{P,Q \in \mathcal{G}\\
      P \preceq R \preceq S \preceq Q}} E(P,Q).
  \end{align}
\end{theorem}
\begin{proof}
  The \eqref{P2} condition implies that $\mathcal{G}$ is an antichain:
  if one member were contained in another, its overlap with the other
  members would have its full measure.

  Fix $(R,S) \in \mathcal{G}^{\ultriangle}$. By the definitions of the
  directional good parts, there are sets $F_1 \subseteq \pi_1(R)$ and
  $F_2 \subseteq \pi_2(S)$ such that
  \begin{align*}
    E_2(R;\, \mathcal{G}) &= F_1 \times \pi_2(R), \\
    E_1(S;\, \mathcal{G}) &= \pi_1(S) \times F_2.
  \end{align*}
  For every $T \in \mathcal{G}$, each $E_i(T;\, \mathcal{G})$ contains
  $T \setminus \sh(\mathcal{G} \setminus \{T\})$. Thus the \eqref{P2} condition gives
  \begin{align*}
    |F_1| \geq \frac{1}{2}|\pi_1(R)|
    \quad\text{and}\quad
    |F_2| \geq \frac{1}{2}|\pi_2(S)|.
  \end{align*}
  Since $R \preceq S$, we have
  $R \cap S = \pi_1(R) \times \pi_2(S)$ and $E(R,S)=F_1\times F_2$.
  Consequently,
  \begin{align*}
    |E(R,S)| = |F_1|\,|F_2|
    \geq \frac{1}{4}|\pi_1(R)|\,|\pi_2(S)|
    = \frac{1}{4}|R \cap S|,
  \end{align*}
  proving \eqref{pwi:large_hiding_spot}.

  To prove disjointness and covering, fix $z \in \sh(\mathcal{G})$ and set
  \begin{align*}
    \mathcal{G}_z := \{T \in \mathcal{G}:\, z \in T\}.
  \end{align*}
  By dyadicity and the antichain property, $\mathcal{G}_z$ is a nonempty
  finite chain under $\preceq$. Let $P_z$ and $Q_z$ be its least and
  greatest elements, respectively.
  For $T \in \mathcal{G}_z$, the members of $\xi_2(T;\, \mathcal{G})$
  containing $z$ are precisely the elements of $\mathcal{G}_z$ strictly
  below $T$; similarly, $\xi_1(T;\, \mathcal{G})$ corresponds to those
  strictly above $T$. Hence, for every $T \in \mathcal{G}$,
  \begin{align*}
    z \in E_2(T;\, \mathcal{G}) &\iff T=P_z, \\
    z \in E_1(T;\, \mathcal{G}) &\iff T=Q_z.
  \end{align*}
  It follows that $z$ belongs to exactly one of the sets $E(P,Q)$,
  namely $E(P_z,Q_z)$. This proves disjointness and covering, and together
  with \eqref{pwi:large_hiding_spot} proves the sparsity assertion.

  Finally, fix $(R,S) \in \mathcal{G}^{\ultriangle}$.
  If $z \in R \cap S$, then
  $P_z \preceq R \preceq S \preceq Q_z$, so $z \in E(P_z,Q_z)$ belongs
  to the right-hand side of \eqref{pwi:cover}.
  Conversely, if $P \preceq R \preceq S \preceq Q$, then
  \begin{align*}
    E(P,Q) &\subseteq P \cap Q = \pi_1(P) \times \pi_2(Q) \\
    &\subseteq \pi_1(R) \times \pi_2(S) = R \cap S.
  \end{align*}
  This proves \eqref{pwi:cover}. Figure~\ref{fig:pwi_geometry}
  illustrates the directional good parts.
\end{proof}

\begin{figure}[htpb]
  \centering
  \begin{tikzpicture}[x=0.23cm,y=0.23cm,
    lbl/.style={font=\small,inner sep=1.25pt}]
    \node[lbl] at (7.5,19) {Directional good parts};
    \fill[gray!10] (0,0) rectangle (4,16);
    \fill[gray!10] (0,0) rectangle (16,4);
    \fill[figteal!30] (1,0) rectangle (4,16);
    \fill[figplum!30] (0,1) rectangle (16,4);
    \fill[figslate!45] (1,1) rectangle (4,4);
    \draw[gray!55,densely dashed,thin] (1,0) -- (1,16);
    \draw[gray!55,densely dashed,thin] (0,1) -- (16,1);
    \draw[figteal,very thick] (0,0) rectangle (4,16);
    \draw[figplum,very thick] (0,0) rectangle (16,4);
    \node[figteal,lbl] at (2,17.2) {$R$};
    \node[figplum,lbl] at (17.2,2) {$S$};
    \node[figteal!75!black,lbl,rotate=90] at (2.5,10)
      {$E_2(R)$};
    \node[figplum!80!black,lbl] at (10,2.5) {$E_1(S)$};
    \draw[figteal,thick] (1,-1.3) -- (4,-1.3);
    \draw[figteal,thick] (1,-0.9) -- (1,-1.7);
    \draw[figteal,thick] (4,-0.9) -- (4,-1.7);
    \node[figteal,lbl] at (2.5,-2.8) {$F_1$};
    \draw[figplum,thick] (-1.3,1) -- (-1.3,4);
    \draw[figplum,thick] (-0.9,1) -- (-1.7,1);
    \draw[figplum,thick] (-0.9,4) -- (-1.7,4);
    \node[figplum,lbl] at (-3.1,2.5) {$F_2$};
    \node[lbl] at (8,-5.1) {$E(R,S)=F_1\times F_2$};
  \end{tikzpicture}
  \caption{The directional good parts for $R\preceq S$.
    The teal and plum regions are $E_2(R)=F_1\times\pi_2(R)$ and
    $E_1(S)=\pi_1(S)\times F_2$, respectively.
    Their slate-colored intersection is $E(R,S)=F_1\times F_2$.
    The \eqref{P2} condition ensures that this set occupies at least one quarter
    of the area of $R\cap S$.}
  \label{fig:pwi_geometry}
\end{figure}
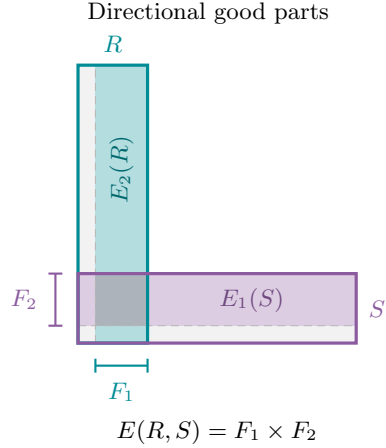

Let $\{a_m\}$ be a sequence of non-negative numbers, and for an interval
$[r,s]$ of indices, define the \emph{mass} of $a$ over $[r,s]$, denoted by
$M_{[r,s]}$, to be the sum of the terms in the sequence $a$ over such interval:
\begin{align*}
  M_{[r,s]} = \sum_{m \in [r,s]} a_m.
\end{align*}
The next lemma relates the sum of the $q$-th power of the mass of $a$
over all possible intervals. It is an estimate of the form
\begin{align*}
  \sum_{\substack{r,s\\ 1 \leq r \leq s \leq N}} M_{[r,s]}^q \gtrsim
    \frac{\|a\|_{\ell^1}^\alpha}{\|a\|_{\ell^q}^\beta}.
\end{align*}
A simple scaling argument shows that, if this inequality were to hold, then
we would have $q = \alpha - \beta$.
Similarly, if the sequence consists of the same number repeated $N$ times,
and the inequality holds with a constant independent of $N$, then
\begin{align*}
  q + 2 \geq \alpha - \frac{\beta}{q}.
\end{align*}
At the endpoint, we thus have the possible candidate inequality
\begin{align*}
  \sum_{\substack{r,s\\ 1 \leq r \leq s \leq N}} M_{[r,s]}^q \gtrsim
    \frac{\|a\|_{\ell^1}^{q+2q'}}{\|a\|_{\ell^q}^{2q'}}.
\end{align*}
This motivates the following lemma.

\begin{lemma}\label{lem:interval_mass}
  Let $\{a_m\}_{m=1}^N$ be a sequence of non-negative numbers,
  not all of them zero. Then for $q > 1$ we have
  \begin{align*}
    \sum_{r \leq s} \Bigg( \sum_{m \in [r,s]} a_m \Bigg)^q \geq
 \frac{1}{2^{q+4q'}} \frac{\|a\|_{\ell^1}^{q + 2q'}}{\|a\|_{\ell^q}^{2q'}}.
  \end{align*}
\end{lemma}
\begin{proof}
  We can assume without loss of generality that $\|a\|_{\ell^1} = 1$.
  Define
  \begin{align*}
    r_* &= \min \bigg\{ r \in [N]:\, \sum_{m \in [1,r]} a_m \geq \frac{1}{4} \bigg\}, \\
    s_* &= \max \bigg\{ s \in [N]:\, \sum_{m \in [s,N]} a_m \geq \frac{1}{4} \bigg\},
  \end{align*}
  where $[N] = \{1, \dots, N\}$.
  By definition we have 
  \begin{align*}
    \sum_{m \in [1, r_*)} a_m < \frac{1}{4} \quad\text{and}\quad
    \sum_{m \in (s_*, N]} a_m < \frac{1}{4}.
  \end{align*}
  
  We have $r_* \leq s_*$ since otherwise summing from $1$ to $s_*$ would
  exceed $\frac{1}{4}$:
  \begin{align*}
    \sum_{m \in [1, s_*]} a_m &= 1 - \sum_{m \in (s_*, N]} a_m \\
    &\geq 1 - \frac{1}{4} > \frac{1}{4},
  \end{align*}
  which violates the minimality of $r_*$.
  Also, the mass over the interval $[r_*, s_*]$ is large:
  \begin{align} \label{magic_lemma:large_mass}
    \sum_{m \in [r_*,s_*]} a_m = 1 - \sum_{m \in [1,r_*) \cup (s_*,N]} a_m
    \geq \frac{1}{2}.
  \end{align}

  By H\"older's inequality we have
  \begin{align*}
    \frac{1}{4} &\leq \sum_{m \in [1,r_*]} a_m \leq \|a\|_{\ell^q} r_*^{\frac{1}{q'}};
  \end{align*}
  thus $r_* \geq (4\|a\|_{\ell^q})^{-q'}$.
  Similarly, $N-s_*+1 \geq (4\|a\|_{\ell^q})^{-q'}$.
  Therefore, the number of intervals containing $[r_*,s_*]$ is at least
  $(4\|a\|_{\ell^q})^{-2q'}$. By \eqref{magic_lemma:large_mass},
  each such interval has mass at least $1/2$, so its contribution to the
  sum below is at least $2^{-q}$. Consequently,
  \begin{align*}
    \sum_{r \leq s} \Bigg( \sum_{m \in [r,s]} a_m \Bigg)^q
    &\geq \frac{1}{2^q(4\|a\|_{\ell^q})^{2q'}} = \frac{1}{2^{q+4q'} \|a\|_{\ell^q}^{2q'}}.
  \end{align*}
\end{proof}

\begin{theorem}\label{weak:dyadic:2dim}
Let $1<p<\infty$, and let $w\in A_p(\mathcal{R}_2)$.
Then the strong maximal operator $\mathcal{M}_{\mathcal{R}_2}$,
taken over the family $\mathcal{R}_2$, satisfies
\[
\|\mathcal{M}_{\mathcal{R}_2}\|_{L^p(w)\to L^{p,\infty}(w)}
\le c_{p} 
[w]_{A_p(\mathcal{R}_2)}^{\frac{1}{p} \left(2+\frac{1}{2(p-1)}\right)}.
\]
\end{theorem}

\begin{proof}
Without loss of generality, assume $f \geq 0$.
Set
\[
q=\frac{p}{p-1},
\qquad
\sigma=w^{1-q}=w^{-\frac1{p-1}},
\qquad
A=[w]_{A_p(\mathcal{R}_2)}.
\]
  
  By monotone convergence, it suffices to prove a uniform estimate for
  $w(\sh(\mathcal{E}))$, where $\mathcal{E}$ is a finite family of rectangles
  satisfying $\langle f \rangle_R > \lambda$ for every $R \in \mathcal{E}$.
  Fix such a family and set $\Omega = \sh(\mathcal{E})$.

  Apply Corollary~\ref{p2_weighted_selection} to $\mathcal{E}$ to obtain a \eqref{P2}
  family $\mathcal{F} \subseteq \mathcal{E}$ such that
  \begin{align*}
    w(\Omega) = w(\sh(\mathcal{E})) \lesssim [w]_{A_p(\mathcal{R}_2)} w(\sh(\mathcal{F})).
  \end{align*}

  We now proceed as follows:
  \begin{align*}
    \lambda^p w(\sh(\mathcal{F})) &\leq \sum_{R \in \mathcal{F}} w(R) \lambda^p \\
    &\leq \sum_{R \in \mathcal{F}} \frac{w(R)}{|R|^p} \Big(\int_R f \Big)^p \\
    &= \sum_{R \in \mathcal{F}} \frac{w(R)(\sigma(R))^{p-1}}{|R|^p} \langle f, \frac{\ind_R}{(\sigma(R))^{1/q}} \rangle^p \\
    &\leq [w]_{A_p(\mathcal{R}_2)} \sum_{R \in \mathcal{F}} \langle g, e_R \rangle_{L^p(\sigma)}^p,
  \end{align*}
  where $g = f/\sigma = fw^{1/(p-1)}$ and $e_R = \frac{\ind_R}{(\sigma(R))^{1/q}}$.
  In particular, $\|g\|_{L^p(\sigma)} = \|f\|_{L^p(w)}$.

  Let $T : L^p(\sigma) \to \ell^p(\mathcal{F})$ be given by
  \begin{align*}
    (Tg)_R = \langle g, e_R \rangle_{L^p(\sigma)}.
  \end{align*}

  In this setting, we thus have
  \begin{align*}
    \lambda^p w(\Omega) \lesssim [w]_{A_p(\mathcal{R}_2)}^2 \|T\|_{L^p(\sigma) \to \ell^p(\mathcal{F})}^p \|f\|_{L^p(w)}^p.
  \end{align*}
  
  The adjoint, $T^* : \ell^q(\mathcal{F}) \to L^q(\sigma)$, is given by
  \begin{align*}
    T^*a = \sum_{R \in \mathcal{F}} a_R e_R.
  \end{align*}
  Therefore, $\|T\| = \|T^*\|$.

  By positivity, it suffices to find $B$ such that, for all $a_R \geq 0$,
  \begin{align*}
    \int \bigg( \sum_{R \in \mathcal{F}} a_R e_R \bigg)^q d\sigma
    \leq B \sum_{R \in \mathcal{F}} a_R^q,
  \end{align*}
  or, equivalently, for all $c_R \geq 0$,
  \begin{align*}
    \int \bigg( \sum_{R \in \mathcal{F}} c_R \ind_R \bigg)^q d\sigma
    \leq B \sum_{R \in \mathcal{F}} c_R^q \sigma(R).
  \end{align*}

  For $P,Q \in \mathcal{F}$ with $P \preceq Q$, let $E(P,Q)$ be the sparse
  part of $P \cap Q$ provided by Theorem~\ref{pwi:sparse_intersections}, and write
  \begin{align*}
    [P,Q] := \{R \in \mathcal{F}:\, P \preceq R \preceq Q\}.
  \end{align*}
  This order interval is a finite chain under $\preceq$: every member
  contains $P \cap Q$, so dyadicity and the antichain property of
  $\mathcal{F}$ imply that any two members are comparable.
  Using the disjoint covering from Theorem~\ref{pwi:sparse_intersections}
  and \eqref{pwi:cover} with $R=S$, we have
  \begin{align*}
    \int \bigg( \sum_{R \in \mathcal{F}} c_R \ind_R \bigg)^q d\sigma
    &= \sum_{\substack{P,Q \\ P \preceq Q}} \int_{E(P,Q)}
        \bigg( \sum_{R \in \mathcal{F}} c_R \ind_R \bigg)^q d\sigma \\
    &= \sum_{\substack{P,Q \\ P \preceq Q}}
        \sigma(E(P,Q)) \bigg( \sum_{\substack{R \in [P,Q]}} c_R \bigg)^q.
  \end{align*}
  Similarly
  \begin{align*}
    \sum_{R \in \mathcal{F}} c_R^q \sigma(R) &=
        \int \sum_{R \in \mathcal{F}} c_R^q \ind_R \, d\sigma \\
    &= \sum_{\substack{P,Q \\ P \preceq Q}} \sigma(E(P,Q))
        \sum_{R \in [P,Q]} c_R^q.
  \end{align*}
  So we need to find a bound $B$ such that
  \begin{align*}
    X:= \sum_{\substack{P,Q \\ P \preceq Q}} 
        \sigma(E(P,Q)) \bigg( \sum_{\substack{R \in [P,Q]}} c_R \bigg)^q
    \leq B
    \underbrace{\sum_{\substack{P,Q \\ P \preceq Q}} \sigma(E(P,Q))
        \sum_{R \in [P,Q]} c_R^q}_{Y}.
  \end{align*}

  In what follows, fractions with zero denominator are defined to be zero,
  since all coefficients in the corresponding interval then vanish.
  By H\"older's inequality we have
  \begin{align*}
    X 
    &= \sum_{\substack{P,Q \\ P \preceq Q}} 
        \sigma(E(P,Q)) \frac{\bigg( \sum_{\substack{R \in [P,Q]}} c_R \bigg)^q}{\bigg(\sum_{R \in [P,Q]} c^q_R \bigg)^{\frac{2}{(q+1)}}}
        \bigg(\sum_{R \in [P,Q]} c^q_R \bigg)^{\frac{2}{(q+1)}} \\
    &\leq
        \left(
            \sum_{\substack{P,Q \\ P \preceq Q}} 
            \sigma(E(P,Q)) \frac{\bigg( \sum_{\substack{R \in [P,Q]}} c_R \bigg)^{q+2q'}}{\bigg(\sum_{R \in [P,Q]} c^q_R \bigg)^{\frac{2q'}{q}}}    
        \right)^{\frac{q-1}{q+1}}
        Y^{\frac{2}{(q+1)}}.
  \end{align*}

  Fix $P \preceq Q$. Applying Lemma~\ref{lem:interval_mass} to the chain
  $[P,Q]$, we have
  \begin{align*}
    \frac{\bigg( \sum_{\substack{R \in [P,Q]}} c_R \bigg)^{q+2q'}}{\bigg(\sum_{R \in [P,Q]} c^q_R \bigg)^{\frac{2q'}{q}}} &\lesssim \sum_{\substack{R,S \in \mathcal{F}\\
    P \preceq R \preceq S \preceq Q}} \Bigg( \sum_{K \in [R,S]} c_K \Bigg)^q. 
  \end{align*}
  Thus, after exchanging the order of summation, we have
  \begin{align*}
    \sum_{\substack{P,Q \\ P \preceq Q}} \sigma(E(P,Q))
    \sum_{\substack{R,S \in \mathcal{F}\\
    P \preceq R \preceq S \preceq Q}} \Bigg( \sum_{K \in [R,S]} c_K \Bigg)^q
    &= \sum_{\substack{R,S \in \mathcal{F} \\ R \preceq S}}
    \Bigg( \sum_{K \in [R,S]} c_K \Bigg)^q
    \sum_{\substack{P,Q \in \mathcal{F} \\ P \preceq R \\ S \preceq Q }} \sigma(E(P,Q)) \\
    &= \sum_{R \preceq S} \Bigg( \sum_{K \in [R,S]} c_K \Bigg)^q
    \sigma(R \cap S) \\
    &\lesssim [w]^{q-1}_{A_p(\mathcal{R}_2)} \sum_{R \preceq S} \Bigg( \sum_{K \in [R,S]} c_K \Bigg)^q
    \sigma(E(R,S)) \\
    &= [w]^{q-1}_{A_p(\mathcal{R}_2)} X.
  \end{align*}
  Here the second equality uses \eqref{pwi:cover}, and the inequality follows
  from \eqref{pwi:large_hiding_spot}, H\"older's inequality, and
  $[\sigma]_{A_q(\mathcal{R}_2)}=[w]_{A_p(\mathcal{R}_2)}^{q-1}$.

  Combining the above inequalities we arrive at
  \begin{align*}
    X \lesssim \Big( [w]^{q-1}_{A_p(\mathcal{R}_2)} X \Big)^{\frac{q-1}{q+1}} Y^{\frac{2}{q+1}},
  \end{align*}
  which implies
  \begin{align*}
    X \lesssim [w]_{A_p(\mathcal{R}_2)}^{\frac{(q-1)^2}{2}} Y.
  \end{align*}
  Which leads to
  \begin{align*}
    \lambda^p w(\Omega) \lesssim [w]_{A_p(\mathcal{R}_2)}^{2+\frac{p(q-1)^2}{2q} } \int f^p w.
  \end{align*}
  Since $q=p'$, taking the supremum over finite families $\mathcal{E}$
  and then taking $p$th roots gives the conclusion. 
\end{proof}

As a consequence of the weak-type estimate above, together with standard
interpolation arguments and the reverse H\"older inequality for strong $A_p(\mathcal{R}_2)$ weights obtained by Luque, P\'erez and Rela
\cite{LPR2017}*{Theorem 1.2}, we obtain the following strong-type estimate.
\begin{theorem}\label{strong:dyadic:2dim}
Let $1<p<\infty$ and let $w\in A_p(\mathcal{R}_2)$.
Then the strong maximal operator $\mathcal{M}_{\mathcal{R}_2}$ satisfies
\[
 \|\mathcal{M}_{\mathcal{R}_2}\|_{L^p(w)\to L^p(w)}
 \lesssim_p
 [w]_{A_p(\mathcal{R}_2)}^{\,\frac1p\left(2+\frac{3}{2(p-1)}\right)}.
\]
\end{theorem}
\begin{proof}
For the sake of completeness, we include the short argument. Set
\[
A=[w]_{A_p(\mathcal{R}_2)},
\qquad
\sigma=w^{-\frac1{p-1}}.
\]
By the reverse H\"older inequality and the quantitative open property for
strong weights proved in \cite{LPR2017} applied to $\sigma$, there exists
\[
0<\varepsilon\leq\frac{p-1}{2},
\qquad
\varepsilon\sim_p A^{-1/(p-1)}
\]
such that, with $r=p-\varepsilon$,
\[
w\in A_r(\mathcal{R}_2),
\qquad
[w]_{A_r(\mathcal{R}_2)}\lesssim_p A.
\]
Applying the weak-type estimate at the exponent $r$ gives
\[
\|\mathcal{M}_{\mathcal{R}_2}\|_{L^r(w)\to L^{r,\infty}(w)}
\lesssim_p
A^{\,\frac1r\left(2+\frac{1}{2(r-1)}\right)}.
\]
Since $\|\mathcal{M}_{\mathcal{R}_2}\|_{L^\infty\to L^\infty}=1$, Marcinkiewicz interpolation yields
\[
\|\mathcal{M}_{\mathcal{R}_2}\|_{L^p(w)\to L^p(w)}
\lesssim_p
\varepsilon^{-1/p}
A^{\,\frac1p\left(2+\frac{1}{2(r-1)}\right)}.
\]
Since $r-1\geq(p-1)/2$ and $\varepsilon\lesssim_p A^{-1/(p-1)}$, we have
\begin{align*}
 A^{\frac{1}{2p}\left(\frac{1}{r-1}-\frac{1}{p-1}\right)}
 &= \exp\left(\frac{\varepsilon\log A}{2p(r-1)(p-1)}\right) \\
 &\leq \exp\left(C_p A^{-1/(p-1)}\log A\right)\lesssim_p 1,
\end{align*}
since $A\geq1$ and $A^{-1/(p-1)}\log A$ is bounded in terms of $p$.
Using also $\varepsilon^{-1/p}\lesssim_p A^{1/(p(p-1))}$, we obtain
\[
\|\mathcal{M}_{\mathcal{R}_2}\|_{L^p(w)\to L^p(w)}
\lesssim_p
A^{\,\frac{2}{p}
+\frac{1}{2p(p-1)}
+\frac{1}{p(p-1)}}
=
A^{\,\frac1p\left(2+\frac{3}{2(p-1)}\right)}.
\]
\end{proof}
Let us finally point out that the two-dimensional estimate above also
improves, by a simple iteration argument, the previously known bounds
in every dimension. Recall that the standard argument consists in
iterating the one-dimensional maximal operator, which gives
\[
\|\mathcal{M}_{\mathcal{R}_d}\|_{L^p(w)\to L^p(w)}
\lesssim_{p,d}
[w]_{A_p(\mathcal{R}_d)}^{\,d/(p-1)}.
\]
For instance, when $p=2$ this gives the exponent $2$ in dimension two,
while our estimate gives $7/4$.
Instead of iterating one coordinate at a time, we can group the
coordinates in pairs and apply the two-dimensional estimate to each
pair.
More precisely, the dyadic $d$-dimensional strong maximal operator is
pointwise dominated by the composition of
$\lfloor d/2\rfloor$ two-dimensional dyadic strong maximal operators, together
with one one-dimensional dyadic maximal operator when $d$ is odd. The
corresponding restrictions of $w$ to coordinate slices have strong
$A_p$ constant bounded by $[w]_{A_p(\mathcal{R}_d)}$.
Therefore we may use
\[
\|\mathcal{M}_{\mathcal{R}_2}\|_{L^p(w)\to L^p(w)}
\lesssim_p
[w]_{A_p(\mathcal{R}_2)}^{\,\frac1p
\left(2+\frac{3}{2(p-1)}\right)}
\]
for each pair of coordinates and, if $d$ is odd, the classical
one-dimensional estimate
\[
\|\mathcal{M}_{\mathcal{R}_1}\|_{L^p(w)\to L^p(w)}
\lesssim_p
[w]_{A_p(\mathcal{R}_1)}^{1/(p-1)}.
\]
This immediately yields the following improvement in every dimension.
\begin{corollary} \label{strong:dyadic:all_dims}
Let $1<p<\infty$, $d\ge2$, and let $w\in A_p(\mathcal{R}_d)$. Then
\[
\|\mathcal{M}_{\mathcal{R}_d}\|_{L^p(w)\to L^p(w)}
\lesssim_{p,d}
[w]_{A_p(\mathcal{R}_d)}^{\alpha},
\]
where
\[
    \alpha = \frac{d}{p-1} - \frac{\Big\lfloor \dfrac{d}{2}\Big\rfloor}{2p(p-1)}.
\]
\end{corollary}

The weak-type inequality also extends to higher dimensions by placing one
two-dimensional maximal operator outermost, applying the weak-type estimate
to that operator, and using strong-type estimates for the remaining
coordinate blocks. 

\begin{corollary} \label{weak:dyadic:all_dims}
  Let $1<p<\infty$, $d\geq 2$, and $w\in A_p(\mathcal{R}_d)$. Then
  \begin{align*}
    \|\mathcal{M}_{\mathcal{R}_d}f\|_{L^{p,\infty}(w)}
    \lesssim_{p,d} [w]_{A_p(\mathcal{R}_d)}^{\beta}
    \|f\|_{L^p(w)},
  \end{align*}
  where
  \begin{align*}
    \beta = \frac1p + \frac{d-1}{p-1}
      - \frac{\Big\lfloor \dfrac{d}{2} \Big\rfloor}{2p(p-1)}.
  \end{align*}
\end{corollary}
\begin{proof}
Write $d=2k+\ell$, where $k=\lfloor d/2\rfloor$ and
$\ell\in\{0,1\}$, and set $A=[w]_{A_p(\mathcal{R}_d)}$.
Use the same coordinate decomposition as above, placing one of the
two-dimensional maximal operators outermost. Apply the weak-type estimate
of Theorem \ref{weak:dyadic:2dim} to this operator on each coordinate slice.
The $A_p$ characteristics of the sliced weight are bounded by $A$, so integrating the
resulting distribution-function inequality over the remaining coordinates
gives a weak-type estimate on $\mathbb{R}^d$ with the same power of $A$.
Apply Theorem \ref{strong:dyadic:2dim} to the remaining $k-1$ pairs, and use
Buckley's one-dimensional strong-type estimate \cite{Buckley1993} for
the remaining coordinate when $\ell=1$.
Thus, we obtain
\begin{align*}
  \|\mathcal{M}_{\mathcal{R}_d}\|_{L^p(w)\to L^{p,\infty}(w)}
  \lesssim_{p,d} A^{\beta},
\end{align*}
since the sum of the exponents is
\begin{align*}
  &\frac1p\left(2+\frac{1}{2(p-1)}\right)
  +\frac{k-1}{p}\left(2+\frac{3}{2(p-1)}\right)
  +\frac{\ell}{p-1} \\
  &\qquad=\frac1p+\frac{d-1}{p-1}-\frac{k}{2p(p-1)}
  =\beta.
\end{align*}
\end{proof}

The proofs of the non-dyadic Theorem \ref{MainTheorem}
and Theorem \ref{MainWeakTheorem} follow from Corollary \ref{strong:dyadic:all_dims} and Corollary \ref{weak:dyadic:all_dims}
respectively, applying the higher-dimensional generalization of the dyadic domination theorem of \cite{LiPipherWard2015}*{Theorem 6.1(ii)}.

\section*{Acknowledgments and AI usage}
We thank Andrei K. Lerner for sharing his results with us.
The second author thanks the first for his kind invitation and hospitality while visiting Bah\'ia Blanca.
An LLM was used for minor grammar corrections,
the TikZ code of the figures,
as well as the idea to use (and the proof of)
Lemma \ref{lem:interval_mass}.

\bibliography{bibliography}
\end{document}